\documentclass[11pt,draft,amsmath,amscd,amsbsy,amssymb]{amsart}
\usepackage{amsmath,amsthm,latexsym,amscd,amsbsy,amssymb}
 \relax

\chardef\bslash=`\\ 

\makeatletter \def\verbatim{\interlinepenalty\@M \@verbatim
\leftskip\@totalleftmargin\advance\leftskip2pc
\frenchspacing\@vobeyspaces \@xverbatim} \makeatother \hfuzz1pc

\makeatletter \def\dgt@k{\dg@DX=-3 \dg@DY=2 \dg@SIZE=3}
\makeatother

\makeatletter \def\dgt@kk{\dg@DX=3 \dg@DY=-1 \dg@SIZE=3}

\theoremstyle{plain} \newtheorem{thm}{Theorem}[section]
\theoremstyle{plain} \newtheorem{theorem}{Theorem}[section]
\newtheorem{cor}[thm]{Corollary}
\newtheorem{lemma}[thm]{Lemma}
\newtheorem{prop}[thm]{Proposition}

\theoremstyle{definition} \newtheorem{rem}[thm]{Remark}

\newtheorem{defin}[thm]{Definition}

\usepackage[all]{xy}

\begin{document}

\title[Spaces of max-min measures]
{Spaces of max-min measures \\ on compact Hausdorff spaces}

\author[V. Brydun]{Viktoriya Brydun}
\address{Department of Mechanics and Mathematics,
Lviv National University, Universytetska Str. 1, 79000 Lviv, Ukraine}

\email{v\_frider@yahoo.com}

\author[M. Zarichnyi]{Mykhailo Zarichnyi}
\address{Faculty of Mathematics and Natural Sciences,
University of Rzesz\'ow, 1 Prof. St. Pigo\'n Street
35-310, Rze\-sz\'ow, Poland}

\email{zarichnyi@yahoo.com}

\thanks{The authors are indebted to the referees for their important remarks.}



\begin{abstract} The notion of max-min measure is a counterpart of the notion of max-plus measure (Maslov measure or idempotent measure). In this paper we consider the spaces of max-min measures on the compact Hausdorff spaces. It is proved that the obtained functor of max-min measures is isomorphic to the functor of max-plus (idempotent) measures considered by the second-named author. However, it turns out that the monads generated by these functors are not isomorphic.
\end{abstract}

\maketitle

\section{Introduction}

The non-additive measures find their applications in different parts of mathematics as well as in mathematical economics, image processing, fractal geometry, optimization etc. Some classes of non-additive measures (in particular, Maslov measures) belong to the idempotent mathematics \cite{KN}. Recall that the latter is a part of mathematics in which the ordinary arithmetic operations are replaced by idempotent ones (see, e.g., \cite{KN,LM}). According to the Correspondence Principle \cite{LM}, to every interesting notion or result of ordinary mathematics there corresponds an interesting notion or result of the idempotent mathematics.

By the Riesz-Markov-Kakutani representation theorem, there is a one-to-one correspondence between the normed positive linear functionals on $C(X)$ and the regular Borel probability measures in $X$, where $X$ is a compact Hausdorff space. In \cite{Z} the second-named author considered an idempotent counterpart of the probability measures, namely, the functor of idempotent measures in the category $\mathbf{Comp}$ of compact Hausdorff spaces and continuous maps. In particular, it was proved in \cite{Z} that this functor is open (i.e, preserves the class of open maps) and generates a monad in the category $\mathbf{Comp}$. It is also proved that the functors of idempotent measures and probability measures are not isomorphic.

The present paper is devoted to another class of measures in Idempotent Mathematics, namely, the class of the max-min measures. In \cite{CRZ}, the max-min measures of finite and compact support are considered on the ultrametric spaces. The definition of the max-min measure is essentially that of the Sugeno integral with respect to a non-additive (idempotent) measure \cite{Su}. Using the term ``max-min measure'' seems to be an abusing of the terminology. We follow the same term in the present paper, in which we develop the theory of max-min measures for the class of compact Hausdorff spaces.

One of the main results of the paper is that the functors max-plus measures and max-min measures are isomorphic.
Some unexpectedness of this isomorphism lies in a  substantial difference between the max-plus and max-min measures:
the latter are not defined to be continuous (as functionals on the suitable Banach space of continuous functions) and establishing
their continuity is not an easy procedure. To prove the existence of this  isomorphism we consider a construction inspired by the notion
of density of an idempotent measure. Actually, this leads to an alternative description of the spaces of idempotent measures and, similarly, of max-min measures.

The mentioned isomorphism allows to claim the normality (in the sense of E. Shchepin \cite{S}) of the functor of max-plus measures.

Similarly to the functor of max-plus measures, the functor of max-min measures  naturally generates a monad in the category $\mathbf{Comp}$. However, it turns out that the monads generated by these functors are not isomorphic.

The monad structure allows to establish some connections between the max-plus measures and max-plus convex sets (\cite{Z}; see, e.g., \cite{CGQS} for the backgrounds of the max-plus convexity).
\section{Preliminaries}
A space is a topological space. All maps are assumed to be continuous unless it is explicitly indicated that the continuity of the map considered
requires  verification. By $\mathrm{Cl}\, A$ (resp. $\mathrm{Int}\, A$) we denote the closure (resp. interior) of a set $A$ in a topological space.

Recall that a space is called zero-dimensional if there is a base of its topology consisting of sets that are simultaneously open and closed.

We endow $[-\infty,\infty]=\mathbb R\cup\{-\infty,\infty\}$ with the order topology.

Given a topological space $X$, by $C(X)$ we denote the Banach space of continuous real-valued functions on $X$ (with respect to the $\sup$-norm). If $\varphi,\psi\in C(X)$, by $\varphi\vee \psi$ (respectively $\varphi\wedge \psi$) we denote the pointwise maximum (respectively minimum) of $\varphi$ and $\psi$. If $c\in \mathbb R$, by $c\wedge\varphi$ the pointwise minimum of $\varphi$ and $c$ is denoted.

 For every $c\in \mathbb R$, we will denote by $c_X$ (or even by $c$ if this does not cause any difficulties) the function on $X$ identically equal to $c$.

\begin{defin} Let $X$ be a compact Hausdorff space. A functional $\mu\colon C(X)\to\mathbb R$ is called a {\em max-min measure} if the following are satisfied:
\begin{enumerate}
\item $\mu(c_X)=c$ for every $c\in\mathbb R$;
\item $\mu(\varphi\vee \psi)=\mu(\varphi)\vee \mu(\psi)$ for every $\varphi,\psi\in C(X)$;
\item $\mu(c\wedge\varphi)=c\wedge\mu(\varphi)$ for every $c\in\mathbb R$ and  every $\varphi\in C(X)$.
\end{enumerate}
\end{defin}

A consequence of (2) is that $\mu(\varphi)\le\mu(\psi)$ whenever $\varphi\le\psi$.

Note that we do not require that  $\mu$ in this definition be continuous. It turns out that the continuity is a consequence of the other properties and we will establish this successively.

We denote by $J(X)$ the set of all max-min measures on $X$. Note that, for every $x\in X$ the Dirac measure $\delta_x$ is an example of a max-min measure. Given $x_1,\dots,x_n\in X$ and $\lambda_1,\dots,\lambda_n\in\mathbb R\cup\{\infty\}$ with $\vee_{i=1}^n\lambda_i=\infty$, define $\mu=\vee_{i=1}^n\lambda_i\wedge \delta_{x_i}\colon C(X)\to\mathbb R$ as follows:

\begin{equation}\label{100} \mu(\varphi)=\vee^n_{i=1}\lambda_i\wedge \varphi(x_i).\end{equation}

\begin{prop} $\mu$ is a max-min measure.

\end{prop}

\begin{proof} We consider the first max-min measure condition  $$\mu(c_X)=\vee^n_{i=1}\lambda_i\wedge c_X(x_i)=\vee^n_{i=1}\lambda_i\wedge c=c.$$

Since $$\vee^n_{i=1}\lambda_i\wedge(\lambda\wedge\varphi(x_i))=\vee^n_{i=1}\lambda_i\wedge\lambda\wedge\varphi(x_i)=
\lambda\wedge(\vee^n_{i=1}\lambda_i\wedge\varphi(x_i)),$$ then $$\vee^n_{i=1}(\lambda_i\wedge\delta_{x_i})(\lambda\wedge\varphi)=\lambda\wedge(\vee^n_{i=1}(\lambda_i\wedge\delta_{x_i})(\varphi)).$$ It follows that $\mu(\lambda\wedge\varphi)=\lambda\wedge\mu(\varphi).$

Finally,
\begin{align*}\vee^n_{i=1}(\lambda_i\wedge\delta_{x_i})(\varphi\vee\psi)=&\vee^n_{i=1}(\lambda_i\wedge(\varphi\vee\psi)(x_i))=
\vee^n_{i=1}((\lambda_i\wedge\varphi(x_i))\vee(\lambda_i\wedge\psi(x_i)))\\=&(\vee^n_{i=1}\lambda_i\wedge \varphi(x_i))\vee(\vee^n_{i=1}\lambda_i\wedge \psi(x_i))\end{align*} and
$$(\vee^n_{i=1}(\lambda_i\wedge\delta_{x_i})(\varphi))\vee(\vee^n_{i=1}(\lambda_i\wedge\delta_{x_i})(\psi))=(\vee^n_{i=1}\lambda_i\wedge \varphi(x_i))\vee(\vee^n_{i=1}\lambda_i\wedge \psi(x_i)),$$ whence $\mu(\varphi\vee\psi)=\mu(\varphi)\vee\mu(\psi).$
\end{proof}
We endow the set $J(X)$ with the weak* topology. A base of this topology consists of the sets of the form $$O(\mu;\varphi_1,\dots,\varphi_n;\varepsilon)=\{\nu\in J(X)\mid |\mu(\varphi_i)-\nu(\varphi_i)|<\varepsilon,\ i=1,\dots,n\},$$
where $\mu\in J(X)$, $\varphi_1,\dots,\varphi_n\in C(X)$, and $\varepsilon>0$.

Denote by $\iota\colon J(X)\to \prod_{\varphi\in C(X)}\mathbb R_{\varphi}$ (here $\mathbb R_{\varphi}$ is a copy of $\mathbb R$) a map defined as follows: $$\iota(\mu)=(\mu(\varphi))_{\varphi\in C(X)}, \ \mu\in J(X).$$

\begin{prop}\label{p:iota} The map $\iota$ is an embedding and its image lies in the compact set $$\prod_{\varphi\in C(X)}[-\|\varphi\|, \|\varphi\|]\subset \prod_{\varphi\in C(X)}\mathbb R_{\varphi}.$$
\end{prop}
\begin{proof} The fact that $\iota$ is an embedding immediately follows from the definition of the weak* topology. Let $\varphi\in C(X)$. Since $-\|\varphi\|\le\|\varphi\|$, we see that $$-\|\varphi\|=\mu(-\|\varphi\|)\le\mu(\varphi)\le\mu(\|\varphi\|)=\|\varphi\|,$$ for every $\mu\in J(X)$. Therefore, $\iota(\mu)\in \prod_{\varphi\in C(X)}[-\|\varphi\|, \|\varphi\|]$, for every $\mu\in J(X)$.
\end{proof}

In the sequel, we identify $J(X)$ with its image $\iota(J(X))$. Also, we regard every $x=(x_\varphi)_{\varphi\in C(X)}$ as a functional on $C(X)$, $x(\varphi)=x_\varphi$, $\varphi\in C(X)$.

\begin{prop} Let $X=\{x_1,\dots,x_n\}$.  Every max-min measure $\mu\in J(X)$ can be represented by formula (\ref{100}) for suitable $\lambda_1,\dots,\lambda_n$.
\end{prop}
\begin{proof} We assume that $x_j\neq x_j$, whenever $i\neq j$. Given $a,b\in \mathbb R$, define $\varphi_i^{a,b}\in C(X)$ by the conditions:  $\varphi_i^{a,b}(x_j)= a$, whenever $j=i$, and $\varphi_i^{a,b}(x_j)= b$ otherwise. Note that $\varphi_i^{a,b}\le\varphi_i^{a',b}$, whenever $a\le a'$, and $\varphi_i^{a,b}\le\varphi_i^{a,b'}$, whenever $b\le b'$.

Let $$\lambda_i^a=\lim_{b\to-\infty}\mu(\varphi_i^{a,b})\in [-\infty,\infty),\ \lambda_i=\lim_{a\to\infty}\lambda^a_i\in[-\infty,\infty].$$

Now let $\varphi\in C(X)$ and let $b\le\min\varphi\le\max\varphi\le a$. Then, clearly, $\varphi(x)=\vee_{i=1}^n\varphi(x_i)\wedge \varphi_i^{a,b}(x)$ and we obtain $\mu(\varphi)=\vee_{i=1}^n\varphi(x_i)\wedge \mu(\varphi_i)$.

Then $$\mu(\varphi)=\lim_{b\to-\infty}\mu(\varphi)=\vee_{i=1}^n\varphi(x_i)\wedge \lim_{b\to-\infty}\mu(\varphi_i^{a,b})=\vee_{i=1}^n\varphi(x_i)\wedge\lambda^a_i$$
and therefore
$$\mu(\varphi)=\lim_{a\to\infty}\mu(\varphi)=\vee_{i=1}^n\varphi(x_i)\wedge \lim_{a\to\infty}\lambda^a_i=\vee_{i=1}^n\varphi(x_i)\wedge\lambda_i.$$

\end{proof}

\begin{cor}\label{c:ineq} For every finite discrete space $X$, every max-min measure $\mu\in J(X)$, and every $c>0$ the inequalities  \begin{equation}\mu(\varphi+c_X)\le\mu(\varphi)+c,\ \mu(\varphi)-c\le \mu(\varphi-c_X)  {\label{101}} \end{equation} hold. \end{cor}
\begin{proof}
 There exist $\lambda_1,\dots,\lambda_n\in\mathbb R\cup\{\infty\}$ with $\vee_{i=1}^n\lambda_i=\infty$ and such that $\mu=\vee_{i=1}^n\lambda_i\wedge \delta_{x_i}$. Then $$\mu(\varphi+c_X)=\vee^n_{i=1}\lambda_i\wedge(\varphi+c_X)(x_i)=\vee^n_{i=1}\lambda_i\wedge(\varphi(x_i)+c)\le\vee^n_{i=1}(\lambda_i\wedge\varphi(x_i)+c)=\mu(\varphi)+c.$$
 The second inequality is an easy consequence of the first one.
\end{proof}

\begin{cor} For every finite discrete space $X$,  every max-min measure $\mu\in J(X)$ is a continuous map on $C(X)$.
\end{cor}
\begin{proof} Let $c>0$. If $\|\varphi-\psi\|\le c$, then $\mu(\psi)\in [\mu(\varphi)-c, \mu(\varphi)+c]$ and the statement follows.
\end{proof}

\begin{prop}\label{p:closed} The set $J(X)$ is closed in the space $\prod_{\varphi\in C(X)}\mathbb R_{\varphi}$.
\end{prop}
\begin{proof} Suppose that $\mu\in \left(\prod_{\varphi\in C(X)}\mathbb R_{\varphi}\right)\setminus J(X)$.

1) If there is $c\in\mathbb R$ such that $\mu(c)\neq c$, then $O(\mu; c_X; |c-\mu(c)|)$ is a neighborhood of $\mu$ that misses $J(X)$.

2) If $\mu(\varphi\vee \psi)\neq \mu(\varphi)\vee\mu(\psi)$, then $$O\left(\mu;\varphi,\psi, \varphi\vee \psi; \frac{|\mu(\varphi\vee \psi)- (\mu(\varphi)\vee\mu(\psi))|}{2}\right)$$ is a neighborhood of $\mu$ that misses $J(X)$.

3) If $\mu(c\wedge\varphi)\neq c\wedge\mu(\varphi)$, then $$O\left(\mu;\varphi,c\wedge\varphi;  \frac{|\mu(c\wedge\varphi)-(c\wedge\mu(\varphi))|}{2}\right)$$ is a neighborhood of $\mu$ that misses $J(X)$.
\end{proof}
\begin{cor} For every compact Hausdorff space $X$, the space $J(X)$ is compact.
\end{cor}
\begin{proof} Indeed, by Propositions \ref{p:iota} and \ref{p:closed} the space $JX)$ can be embedded as a closed subset in the compact Hausdorff space $\prod_{\varphi\in C(X)}[-\|\varphi\|, \|\varphi\|]$ and therefore is compact Hausdorff as well.

\end{proof}

 Let $f\colon X\to Y$ be a continuous map of compact Hausdorff spaces. Given $\mu\in I(X)$, define a map $J(f)(\mu)\colon C(Y)\to\mathbb R$ as follows: $J(f)(\mu)(\varphi)=\mu (\varphi f)$, $\varphi\in C(Y)$.

\begin{prop} Let $f\colon X\to Y$ be a continuous map of compact Hausdorff spaces. Then $J(f)(\mu)\in J(Y)$, for every $\mu\in I(X)$. The obtained map $J(f)\colon J(X)\to J(Y)$ is continuous.
\end{prop}

\begin{proof}  Let $\mu \in J(X)$ and $\varphi, \psi \in C(Y).$ Clearly, $J(f)(\mu)(c_X)=c,\ c\in \mathbb{R},$ and $$J(f)(\mu)(\lambda\wedge \varphi)=\mu((\lambda\wedge\varphi)f)=\lambda \wedge \mu(\varphi f)=\lambda \wedge J(f)(\mu)(\varphi).$$ We have also $$J(f)(\mu)(\varphi\vee \psi)=\mu((\varphi\vee \psi)f)=\mu(\varphi f\vee \psi f)=\mu(\varphi f)\vee \mu(\psi f)= J(f)(\mu)(\varphi)\vee J(f)(\mu)(\psi).$$

Thus, $J(f)(\mu)\in J(Y)$. Therefore, we obtain a map $J(f)\colon J(X)\to J(Y)$.

Let $\mu\in J(X)$, $\psi_1,\dots,\psi_k\in C(Y)$, and $\varepsilon>0$. Then $O(J(f)(\mu);\psi_1,\dots,\psi_n;\varepsilon) $ is a base neighborhood of $J(f)(\mu)$.   Since $$J(f)(\mu)(O(\mu;\psi_1f,\dots,\psi_nf;\varepsilon))\subset O(J(f)(\mu);\psi_1,\dots,\psi_n;\varepsilon) ,$$ we conclude that the map $J(f)$ is continuous.

\end{proof}

It is easy to check that  $J$  is a functor in the category $\mathbf{Comp}$ of compact Hausdorff spaces and continuous maps.

Note that the functor $J$ preserves the class of embeddings. In the sequel, given a closed subspace $A$ of $X$, we will identify the space $J(A)$ with the subspace $J(\iota)((J(A))$ of $J(X)$, where $\iota\colon A\to X$ denotes the inclusion map.

\begin{prop}\label{p:zero} For every zero-dimensional compact  space $X$, for every max-min measure $\mu\in J(X)$, and every $c>0$ the inequalities (\ref{101}) hold.
\end{prop}

\begin{proof}
Now let $X$ be a zero-dimensional space. Suppose that $\mu(\varphi+c_X)>\mu(\varphi)+c $, for some $\mu\in J(X)$ and some $c>0$. Then there is $r>0$ such that $\mu(\varphi+c_X)>\mu(\varphi)+c +r$.

There exists a finite disjoint open cover $\mathcal U$ of $X$ and a function $\psi\in C(X)$ such that
\begin{enumerate}
\item $\psi$ is constant on every element of $\mathcal U$;
\item $\psi\le\varphi\le\psi+(r/2)_X$.
\end{enumerate}
Let $Y=X/\mathcal U$ and let $q\colon X\to Y$ be the quotient map. There is  $\psi'\in C(Y)$ such that $\psi=\psi'q$.

Note that, for every $a>0$,
\begin{align*} \mu(\psi+a_X)=&\mu(\psi'q+a_X)=\mu((\psi'+a_Y)q)=J(q)(\mu)(\psi'+a_Y)\\
\le&J(q)(\mu)(\psi')+a=\mu(\psi'q)+a=\mu(\psi)+a.
\end{align*}

Therefore,
$$\mu(\varphi)+c +r< \mu(\varphi+c_X)\le \mu(\psi+c_X+(r/2)_X)\le \mu(\psi)+c+r/2\le\mu(\varphi)+c+r/2$$
and we obtain a contradiction.

The second inequality is a consequence of the first one (see the proof of Corollary \ref{c:ineq}).

\end{proof}

\begin{cor} Let $X$ be a zero-dimensional compact space and $\mu\in J(X)$. Then $\mu\colon C(X)\to \mathbb R$ is continuous.
\end{cor}

Let $S=\{X_{\alpha},p_{\alpha \beta};\mathcal{A}\}$ be an inverse system over a directed set $\mathcal{A}$.  (See, e.g., \cite{S} for the necessary information concerning inverse systems in the category $\mathbf{Comp}$.) For any $\alpha \in \mathcal{A},$ let $p_{\alpha}\colon X=\lim\limits_{\longleftarrow} S\to X_{\alpha}$ denote the limit projection. By $J(S)$ we denote the inverse system $\{J(X_{\alpha}),J(p_{\alpha \beta});\mathcal{A}\}$.

\begin{prop} \label{p:homeo} Let $X$ be a zero-dimensional compact space and $X=\varprojlim\{X_\alpha,p_{\alpha\beta};\mathcal A\}$, where $\mathcal A$ is a directed set.  Then the natural map $h=(J(p_{\alpha}))_{\alpha \in \mathcal{A}}\colon J(X)\to \lim\limits_{\longleftarrow} J(S)$ is a homeomorphism.
\end{prop}

\begin{proof}  First, we are going to show that the map $h$ is an embedding. Suppose the opposite and let $\mu,\nu\in J(X)$, $\mu\neq\nu$, be such that $h(\mu)=h(\nu)$. Since  $\mu\neq\nu$, there exists $\varphi\in C(X)$ such that $\mu(\varphi)\neq\nu(\varphi)$. Let $C'=\{\varphi p_{\alpha}\, |\, \varphi \in C(X_{\alpha}),\, \alpha \in \mathcal{A}\}.$ Since $C'$ is dense in $C(X)$ and $\mu,\nu$ are continuous, there is $\varphi'\in C'$ such that $\mu(\varphi')\neq\nu(\varphi')$. Then $\varphi'=\psi p_\alpha$, for some $\alpha\in\mathcal A$ and $\psi\in C(X_\alpha)$. Therefore, $$J(p_\alpha)(\mu)(\psi)=\mu(\varphi')\neq\nu(\varphi')=J(p_\alpha)(\nu)(\psi)$$ and we obtain a contradiction.

Now, show that $h$ is an onto map. Let $(\mu_{\alpha})_{\alpha \in \mathcal{A}}\in \lim\limits_{\longleftarrow} J(S)$. We are going to show that there exists $\mu \in J(X)$ such that $J(p_{\alpha})(\mu)=\mu_{\alpha},$  for any $\alpha \in \mathcal{A}.$  Given $\varphi, \psi \in C',$ one can write $\varphi =\varphi'p_{\alpha}, $ $\psi=\psi'p_{\alpha},$ for some $\alpha \in \mathcal{A},$ whence $$\mu(\varphi\vee\psi)=\mu((\varphi'p_{\alpha})\vee(\psi'p_{\alpha}))=J(p_{\alpha})(\mu)(\varphi'\vee\psi')=\mu_{\alpha}(\varphi'\vee\psi')=
\mu_{\alpha}(\varphi')\vee\mu_{\alpha}(\psi')=$$ $$=\mu(\varphi'p_{\alpha})\vee \mu(\psi'p_{\alpha})=\mu(\varphi)\vee\mu(\psi).$$
Since, by the Stone-Weierstrass theorem (see, e.g., \cite{Co}), the set $C'$ is dense in $C(X)$ and the operation $\vee$ is continuous, we conclude that $\mu(\varphi\vee \psi)=\mu(\varphi)\vee\mu(\psi),$ for all $\varphi, \psi \in C(X).$ Similarly, $\mu(\lambda\wedge\varphi)=\lambda\wedge\mu(\varphi)$ for all $\varphi \in C(X)$ and $\lambda \in \mathbb{R}.$ Thus, $\mu\in J(X)$ is as required.

\end{proof}

E. Shchepin (\cite{S} see for details) calls that the just established property of the functor $J$ the zero-dimensional continuity.

In \cite{CRZ}, the tensor product of max-min measures of finite supports is defined. Let $\mu=\vee_{i=1}^n\alpha_i\wedge\delta_{x_i}\in J(X)$, $\nu=\vee_{j=1}^m\beta_j\wedge\delta_{y_j}\in J(Y)$.
Then $$\mu\otimes\nu=\vee_{i=1}^n\vee_{j=1}^m (\alpha_i\wedge\beta_j)\wedge\delta_{(x_i,y_j)}\in J(X\times Y)$$
is called the tensor product of $\mu$ and $\nu$. Similarly, one can define $\mu_1\otimes\dots\otimes\mu_k$.

We are going to define tensor products for arbitrary max-min measures in zero-dimensional compact metrizable spaces. Given $\mu\in J(X)$ and $\nu\in J(Y)$, where $X,Y$ are compact metrizable,
represent $X=\varprojlim \{X_i,p_{ij};\mathbb N\}$,  $Y=\varprojlim \{Y_i,p_{ij};\mathbb N\}$, where $X_i$, $Y_i$ are finite. Denote by $p_i\colon X\to X_i$, $q_i\colon Y\to Y_i$ the limit projections. By Proposition \ref{p:homeo}, there exists a unique $\tau\in J(X\times Y)$ such that $J(p_i\times q_i)(\tau)=J(p_i)(\mu)\otimes J(q_i)(\nu)$. We say that $\tau$ is the tensor product of $\mu$ and $\nu$ and denote it by $\mu\otimes\nu$. Similarly, one can define $\mu_1\otimes\dots\otimes\mu_n$, where $\mu_i\in J(X_i)$ for zero-dimensional compact metrizable $X_i$, $i=1,\dots,n$.

Now, let $T$ be an arbitrary set. Denote by $\mathrm{Fin}\, T$ the family of nonempty finite subsets of $T$.  Suppose that $X=\prod_{\alpha\in T}X_\alpha$, where $X_\alpha$ is a  zero-dimensional compact metrizable space, $\alpha\in T$. Then, clearly, $X=\varprojlim\{\prod_{\alpha\in A}X_\alpha, p_{AB}; \mathrm{Fin}\, T\} $, where $p_{AB}\colon \prod_{\alpha\in A}X_\alpha\to \prod_{\alpha\in B}X_\alpha$ denotes the projection, $A,B\in \mathrm{Fin}\, T$, $A\supset B$. If $\mu_\alpha\in J(X_\alpha)$, then, by Proposition \ref{p:homeo}, there exists a unique $\tau\in J(X)$ such that $J(p_A)(\tau)=\otimes_{\alpha\in A}\mu_\alpha$. Here, $p_A\colon X\to \prod_{\alpha\in A}X_\alpha$ denotes the limit projection, $A\in \mathrm{Fin}\, T$.

\section{Max-min Milyutin maps}

\begin{defin} A map $f\colon X\to Y$ of compact metrizable spaces is called a max-min Milyutin map if there is a map $s\colon Y\to J(X)$ such that $s(y)\in J(f^{-1}(y))\subset J(X)$, for every $y\in Y$.
\end{defin}

\begin{theorem} For every compact Hausdorff space $X$ there exists a  max-min Milyutin map $f\colon Z\to X$, where $Z$ is a zero-dimensional compact Hausdorff  space.
\end{theorem}
\begin{proof}
We first assume that  $X$ is metrizable and some compatible metric on $X$ is chosen. We modify a construction from \cite{AT}.

For every $n\in \mathbb N$, let $\mathcal A_n$ be a finite family of pairs of closed subsets of $X$ satisfying the properties:
\begin{enumerate}
\item $\cup\{B\mid (A,B)\in\mathcal A_n\}=X$;
\item $\mathrm{diam}(A)\le 1/n$, for every $(A,B)\in\mathcal A_n$;
\item $\mathrm{Int}(A)\supset B$, for every $(A,B)\in\mathcal A_n$.
\end{enumerate}

Let $Z_n=\sqcup\{A\mid (A,B)\in\mathcal A_n\}$. Define $f_n\colon Z_n\to X$ by the condition: $f_n|A$ is the inclusion map $\iota_A\colon A\hookrightarrow X$. Then let $$Z=\left\{(z_n)_{n=1}^\infty\in \prod_{n=1}^\infty Z_n\mid f_i(z_i)=f_j(z_j),\text{ for all }i,j\in\mathbb N\right\}.$$

For every $n\in\mathbb N$, let $$Y_n=\left\{(z_m)_{m=1}^n\in \prod_{m=1}^n Z_m\mid f_i(z_i)=f_j(z_j),\text{ for all }i,j\le n\right\}.$$
For $n\ge k$, denote by $g_{nk}\colon Y_n\to Y_k$ the natural projection. Clearly, $Z=\varprojlim\{Y_n,g_{nk};\mathbb N\}$.
It is easy to check that $Z$ is a zero-dimensional space.

For any $(A,B)\in \mathcal A_n$, let $\alpha_{(A,B)}\colon X\to[-\infty,\infty]$ be a continuous function such that $\alpha_{(A,B)}|B=\infty$ and $\alpha_{(A,B)}|(X\setminus A)=-\infty$. Given $x\in X$, define $$\mu_n(x)=\vee_{(A,B)\in \mathcal A_n}\alpha_{(A,B)}(x)\wedge \delta_{\iota_A^{-1}(x)}.$$

Note that $\mu_n(x)$ is well-defined. We are going to show that the map $\mu_n\colon X\to J(Y_n)$ is continuous. Indeed, given $\varphi\in C(X)$, we see that the function $$x\mapsto \mu_n(x)(\varphi)= \vee_{(A,B)\in \mathcal A_n}\alpha_{(A,B)}(x)\wedge \varphi(x)\colon X\to\mathbb R$$
is continuous, and this implies the continuity of $\mu_n$.

Define $f\colon Z\to X$ by the formula $f((z_n)_{n=1}^\infty)=f_1(z_1)$. For every $m\in \mathbb N$, let $h_m\colon Y_m\to X$ be defined by the formula  $h_m((z_n)_{n=1}^m)=f_1(z_1)$.

For any $x\in X$, $f^{-1}(x)=\varprojlim\{h_m^{-1}(x), g_{mk}|h_m^{-1}(x); \mathbb N\}$. By Proposition \ref{p:homeo}, there exists $\mu(x)\in J(f^{-1}(x))$ such that $J(g_m)(\mu(x))=\otimes_{i=1}^m\mu_i(x)$ (by $g_m\colon Z\to Y_m$ we denote the projection map).

Note that the continuity of the map map $x\mapsto \mu(x)$ is a consequence of the continuity of the maps $\mu_n$, $n\in\mathbb N$.

Now, suppose that $X$ is arbitrary. Then one may assume that $X\subset \prod_{\alpha\in T}X_\alpha$, for some family $\{X_\alpha\mid\alpha\in T\}$ of compact metrizable spaces. For every $\alpha\in T$ let $f_\alpha\colon Y_\alpha\to X_\alpha$ be a Milyutin map, where $Y_\alpha$ is a zero-dimensional space. Let $$g=\prod_{\alpha\in T}X_\alpha\colon  \prod_{\alpha\in T}Y_\alpha\to  \prod_{\alpha\in T}X_\alpha.$$
Let $Z=g^{-1}(X)$ and let $f=g|Z\colon Z\to X$. Clearly $Z$ is a zero-dimensional compact Hausdorff space. We are going to show that $f$ is a Milyutin map. For every $\alpha\in T$, let $s_\alpha\colon X_\alpha\to J(Y_\alpha)$ be a map such that $s_\alpha(x)\in J(f_\alpha^{-1}(x))$, for every $x\in X_\alpha$. Define $s((x_\alpha)_{\alpha\in T})=\otimes_{\alpha\in T}s_\alpha(x_\alpha)$.
Clearly, $s$ is continuous and $s(x)\in J(s^{-1}(x))$, for every $x\in X$.

\end{proof}

\begin{prop} If $f\colon Z\to X$ is a max-min Milyutin map, then the map $J(f)\colon J(Z)\to J(Y)$ is onto.
\end{prop}
\begin{proof} Let $s\colon X\to J(Z)$ be a map such that $s(x)\in J(f^{-1}(x))$, for every $x\in X$. Given $\varphi\in C(Z)$, define $\tilde\varphi\colon X\to \mathbb R$ as follows: $\tilde\varphi(x)=s(x)(\varphi)$, $x\in X$.

Note that $\tilde\varphi\in C(X)$. Indeed, let $x_0\in X$ and $\varepsilon>0$. Find a neighborhood $U$ of $x$ such that $s(U)\subset O\langle s(x_0);\varphi;\varepsilon\rangle$. Then $x\in U$ implies $|\tilde\varphi(x)-\tilde\varphi(x_0)|<\varepsilon$.

One may regard $\tilde\varphi$ as the averaging of $\varphi$ (with respect to $s$).

Define $\nu\colon C(Z)\to\mathbb R$ by the formula $\nu(\varphi)=\mu(\tilde\varphi)$, $\varphi\in C(Z)$. One can easily verify that $\nu\in J(Z)$. Since, for every $\varphi\in C(X)$ and every $x\in X$, the restriction of the function $\varphi f$  on every set $f^{-1}(x)$ is constant and equals $\varphi(x)$, we obtain $$J(f)(\nu)(\varphi)=\nu(f \varphi)=\mu(\widetilde{\varphi f})=\mu(\varphi),$$
i.e., $J(f)(\nu)=\mu$.

\end{proof}

\begin{cor} The set $J_\omega(X)$ is dense in $J(X)$.
\end{cor}

Now, one can extend Proposition \ref{p:zero} over the class of all compact Hausdorff spaces.

\begin{prop} For every compact Hausdorff space $X$, for every max-min measure $\mu\in J(X)$, and every $c>0$ the inequalities (\ref{101}) hold.
\end{prop}
\begin{proof} Let $f\colon Z\to X$ be a max-min Milyutin map. Then, given $\mu\in J(X)$, find $\nu\in J(Z)$ such that $J(f)(\nu)=\mu$. Then, for every $\varphi\in C(X)$, $$\mu(\varphi+c)=J(f)(\nu)(\varphi+c)=\nu((\varphi+c)f)=\nu(\varphi f+c)=\nu(\varphi f)+c=\mu(\varphi)+c.$$

The second inequality is a consequence of the first one (see the proof of Corollary \ref{c:ineq}).
\end{proof}

\begin{cor} For every compact Hausdorff space $X$, every $\mu\in J(X)$ is continuous.
\end{cor}

\section{Cones}

In this section we introduce an auxiliary construction which finally will allow us to establish an isomorphism of the functors of max-min measures and idempotent measures (see the definition below). The roots of this construction lie in the possibility of representation of every max-min measure  as a map $\varphi \mapsto \sup(\varphi\wedge g)$, for a suitable function $g$. This is similar to the  notion of density first considered for the idempotent measures; see Remark \ref{r:density} below.

Let $X$ be a set. The cone $\mathrm{Cone}(X)$ is the quotient set $(X\times [0,1])/(X\times\{0\})$. For the sake of simplicity, we denote by $(x,t)$ the equivalence class containing $(x,t)$. Thus, $(x,0)\sim(y,0)$ for all $x,y\in X$. Given $A\subset X$ and $B\subset[0,1]$, we identify $A\times B$ with the subset $\{(a,b)\mid a\in A,\ b\in B\}$ of $\mathrm{Cone}(X)$. Denote by $\eta_X\colon X\to \mathrm{Cone}(X)$ the map sending $x\in X$ to $(x,1)\in \mathrm{Cone}(X)$.

For any topological space $X$ we denote by $\exp X$ the set of all nonempty compact subsets in $X$. The set $\exp X$ is endowed with the Vietoris topology. A base of this topology consists of the sets of the form $$\langle U_1,\dots,U_n\rangle=\{A\in\exp X\mid A\subset\cup_{i=1}^nU_i,\ A\cap U_i\neq\emptyset\text{ for all }i=1,\dots,n\},$$
where $U_1,\dots,U_n$ are open subsets of $X$.

Note that the family of sets of the form $\langle U\rangle=\{A\in \exp X\mid A\subset U\}$ and $\langle X, U\rangle=\{A\in \exp X\mid A\cap U\neq\emptyset\}$, where $U$ is an open subset in $X$, is a subbase of the Vietoris topology of $\exp X$. See, e.g., \cite{Be} for properties of the Vietoris topology.

Let $(X,d)$ be a metric space and $\mathrm{diam}(X)\le1$. We endow $\mathrm{Cone}(X)$ with the following metric $\check d$:
$$\check d((x,s),(y,t))= \min\{s,t\}d(x,y)+|s-t|.$$
Note that the map $\eta_X\colon X\to \mathrm{Cone}(X)$ is an isometric embedding.

 Let $(X,d)$ be a metric space. The set $\exp X$ is endowed with the Hausdorff metric $d_H$,
$$d_H(A,B)=\inf\{r>0\mid A\subset O_r(B),\ B\subset O_r(A)\},$$
where $O_s(K)$ stands for the $s$-neighborhood of $K\subset X$.

A subset $K$ of $\mathrm{Cone}(X)$ is called {\it saturated} if $(x,t)\in K$ implies $(x,t')\in K$ for any $t'\in[0,t]$.

Given a subset $A$ of  $\mathrm{Cone}(X)$, we denote by $\mathrm{Sat}(A)$ the minimal saturated set containing $A$.  By $\bar J_\omega(X)$ we denote the set $\{\mathrm{Sat}(A)\mid A\text{ is finite} \}\subset\bar J(X)$.

Given a metric space $(X,d)$ with $\mathrm{diam}(X)\le1$, let $$\bar J(X)=\{A\in\exp(\mathrm{Cone}(X))\mid A\text{ is saturated and }(x,1)\in A\text{ for some }x\in X \}.$$

\begin{lemma} The set $\bar J(X)$ is a closed subset of $\exp(\mathrm{Cone}(X))$.
\end{lemma}
\begin{proof}
Suppose that $A\notin\bar J(X)$. The proof splits in two cases.

Case 1). $A$ is not saturated. Then there are $x\in X$ and $t',t\in[0,1]$ such that $(x,t)\in A$, $(x,t')\notin A$ and $t'<t$. Without loss of generality one may assume that $0<t'$. Then there exists  a neighborhood $U$ of $x$ in $X$ and disjoint neighborhoods $V$ of $t'$ and $W$ of $t $ respectively such that $0\notin V$ and $(\mathrm{Cl}\,  U\times \mathrm{Cl}\,  V)\cap A=\emptyset$.

Then $$A\in\langle \mathrm{Cone}(X)\setminus (\mathrm{Cl}\,  U\times \mathrm{Cl}\,  V), U\times W\rangle\subset \exp(\mathrm{Cone}(X))\setminus\bar J(X).$$

Case 2). $X\cap A=\emptyset$. Let $U$ be a neighborhood of $A$ such that $\mathrm{Cl}\,  U\cap X=\emptyset$. Then $A\in\langle U\rangle\subset \exp(\mathrm{Cone}(X))\setminus\bar J(X)$.
\end{proof}

Let $\xi\colon [0,1]\to[-\infty,\infty]$ be a continuous order-preserving homeomorphism. For every $A\in\bar J(X)$ define $h_X(A)\colon C(X)\to\mathbb R$ as follows. Given $\varphi\in C(X)$,  let  $$h_X(A)(\varphi)=\vee\{\xi(t)\wedge\varphi(x)\mid (x,t)\in A\}.$$
It is easy to show that $h_X(A)\in J(X)$.

\begin{lemma} Let $\varphi\in C(X)$ and $c\in(0,1]$, then for any $\varepsilon>0$ there is $\delta>0$ such that $d((x,t),(y,s))<\delta$ and $t,s>c$ implies $|\varphi(x)-\varphi(y)|<\varepsilon$.
\end{lemma}
\begin{proof} This follows from the definition of metric on $ \mathrm{Cone}(X)$ and the uniform continuity of $\varphi$.
\end{proof}

\begin{prop}\label{p:homeom} The map $h_X\colon \bar J(X)\to J(X)$ is a homeomorphism.
\end{prop}

\begin{proof} We first show that the map $h_X$ is continuous. Let $(A_i)$ be a convergent sequence in $\bar J(X)$ and $A=\lim_{i\to\infty}A_i$. Given $\varphi\in C(X)$, we have to show that $\lim_{n\to\infty}h_X(A_n)(\varphi)=h_X(A)(\varphi)$.

Note that $h_X$ is an onto map. Indeed, let $\vee_{i=1}^n\lambda_i=\infty$ such that $\mu=\vee_{i=1}^n\lambda_i\wedge \delta_{x_i}\in J_\omega(X)$. Define $$A=\{(x_i,t)\in \mathrm{Cone}(X)\mid i=1,\dots,n,\ t\le\xi^{-1}(\lambda_i)\}.$$
Clearly, $A\in \bar J(X)$ and $h_X(A)=\mu$. Since $\bar J(X)$ is compact and $J_\omega(X)$ is dense in $J(X)$, we conclude that $h_X(\bar J(X))=J(X)$.

Now, let us prove that $h_X$ is an embedding. Let $A,B\in\bar J(X)$ be such that $d_H(A,B)\ge r>0$. There is $(x,t)\in A$ such that $O_r(x,t)\cap B=\emptyset$ and there are neighborhoods $U$ of $x$ and $V$ of $t$ in $[0,1]$ such that $(U\times V)\cap B=\emptyset$. Without loss of generality, we may suppose that $t<1$. Let $\varphi\in C(X)$ be a function such that $\varphi(x)=\xi(x)$ and $\varphi|(X\setminus U)<\xi(x)$. Then $h_X(A)(\varphi)\ge\xi(x)$ and $h_X(B)(\varphi)<\xi(x)$.

\end{proof}

Recall that a max-plus measure on $X$ is a functional $\mu\colon C(X)\to\mathbb R$ satisfying the conditions:
\begin{enumerate}
\item $\mu(c_X)=c$, $c\in\mathbb R$;
\item $\mu(\varphi\vee\psi)=\mu(\varphi)\vee\mu(\psi)$;
\item $\mu(c+\varphi)=c+\mu(\varphi)$
\end{enumerate}
(see, e.g., \cite{Z}; remark that in \cite{Z} as well as in another papers devoted to the max-plus measures $\oplus$ is used for maximum and $\odot$ for addition).

One can similarly define a map $g_X\colon \bar J(X)\to I(X)$, where $I(X)$ denotes the set of idempotent measures (see \cite{Z}). More precisely, if $A\in \bar J(X)$, let
$$g_X(A)(\varphi)=\vee \{\varphi(x)+\ln t\mid (x,t)\in A\}.$$

The following proposition can be proved similarly as Proposition \ref{p:homeom}.

\begin{prop} The map $g_X\colon \bar J(X)\to I(X)$ is a homeomorphism.
\end{prop}

\begin{cor} The spaces $J(X)$ and $I(X)$ are homeomorphic.
\end{cor}

We will show even more, namely, that the functors $J$ and $I$ are isomorphic. Given a map $f\colon X\to Y$, we will first show that the diagram $$\xymatrix{\bar J(X)\ar[r]^{h_X}\ar[d]_{\bar J(f)} & J(X)\ar[d]^{J(f)}\\
\bar J(Y)\ar[r]_{h_Y}& J(Y)}$$
is commutative. Let $A\in \bar J(X)$ and $\varphi\in C(Y)$, then \begin{align*}(J(f)h_X(A))(\varphi)=&h_X(A)(\varphi f)=\vee\{\varphi f(x)\wedge\xi(t)\mid (x,t)\in A\}\\ =&\vee \{\varphi(y)\wedge\xi(t)\mid (y,t)\in\bar J(f)(A)\}=h_Y\bar J(f)(A)(\varphi).\end{align*}
This shows that $h=(h_X)$ is a natural transformation. One can similarly prove that $g=(g_X)$ is a natural transformation.  Therefore, we obtain the following
\begin{prop} \label{p:iso} The functors $I$ and $J$ are isomorphic.
\end{prop}

\begin{rem}\label{r:density} M. Akian \cite{Ak} considered in details the notion of density of an idempotent (max-plus) measure. The density is shown to be an upper-semicontinuous function defined on the space under consideration.  For the reasons of topologization (metrization), it is more convenient to consider the subgraph of this function. Then the densities turn to be elements of the hyperspace of the considered space multiplied by the segment $[-\infty,0]$. However, in this way we do not obtain a subfunctor of the functor $\exp((-)\times[0,\infty])$, as the image of the subgraph of a function is not necessarily the subgraph of a function (defined on the whole space).  That is why we have modified the construction and passed to the hyperspace of the cone.
\end{rem}

\begin{rem} Proposition \ref{p:iso} allows us to claim that the functor $J$ is normal in the sense of Shchepin \cite{S}; in particular, $J$ is continuous, i.e., one can drop the condition of zero-dimensionality in Proposition  \ref{p:homeo}.
\end{rem}

\section{Monads}
A monad on a category $\mathcal C$ is a triple $\mathbb T=(T,\eta,\psi)$, where $T\colon  \mathcal C\to\mathcal C$ is a functor,
$\eta\colon 1_{\mathcal C}\to T$ (unit), $\psi\colon T^2=TT\to T$ (multiplication) are natural transformations satisfying the properties:
$\psi T(\eta)=\psi \eta_T=1_T$ (two-side unit), $\psi\psi_T=\psi T(\psi)$ (associativity).

Two monads,  $\mathbb T=(T,\eta,\psi)$, $\mathbb T'=(T',\eta',\psi')$, on a category $\mathcal C$ are {\em isomorphic} if there exists an isomorphism
of functors $k\colon T\to T'$ such that $k\eta=\eta'$ and $k\psi=\psi'k_{T'}T(k)=\psi'T'(k)k_T$.

See, e.g., \cite{BW} for detailed exposition of the monad theory.

For any $X$ define $\eta_X\colon X\to J(X)$ by the formula $\eta_X(x)=\delta_x$.
\begin{prop} The map $\eta_X$ is continuous.
\end{prop}
\begin{proof} The assertion is a consequence of the formula  $$\eta_X^{-1}(O(\delta_x;\varphi_1,\dots,\varphi_n;\varepsilon))=\bigcap_{i=1}^n\varphi_i^{-1}((\varphi_i(x)-\varepsilon,\varphi_i(x)+\varepsilon)).$$
\end{proof}
Given $\varphi\in C(X)$, define $\bar\varphi\colon J(X)\to\mathbb R$ by the formula $\bar\varphi(\mu)=\mu(\varphi)$, $\mu\in J(X)$.
\begin{lemma} $\bar\varphi\in C(J(X))$.
\end{lemma}
\begin{proof} This follows from the formula $$\bar\varphi^{-1}((\bar\varphi(\mu)-\varepsilon, \bar\varphi(\mu)+\varepsilon))=O(\mu;\varphi;\varepsilon),\ \mu\in J(X).$$
\end{proof}

Given $M\in J^2(X)$, define $\psi_X(M)\colon C(X)\to \mathbb R$ by the formula $\psi_X(M)(\varphi)=M(\varphi)$, $\varphi\in C(X)$.
\begin{prop} $\psi_X(M)\in J(X)$.
\end{prop}
\begin{proof}
 1) Clearly, $\psi_X(M)(c)=M(c)=c$, for every $c\in\mathbb R$.

2) Since, for all $\varphi_1,\varphi_2\in C(X)$,  $\bar\varphi_1\vee\bar\varphi_2=\overline{\varphi_1\vee\varphi_2}$, we see that \begin{align*}\psi_X(M)(\varphi_1\vee\varphi_2)=&M(\overline{\varphi_1\vee\varphi_2})=M(\bar\varphi_1\vee\bar\varphi_2)=M(\bar\varphi_1)\vee M(\bar\varphi_2)\\ =&\psi_X(M)(\varphi_1)\vee\psi_X(M)(\varphi_2).\end{align*}

3) Since $\overline{\lambda\wedge\varphi}=\lambda\wedge\bar\varphi$, we see that
$$\psi_X(M)(\lambda\wedge\varphi)=M(\overline{\lambda\wedge\varphi})=M(\lambda\wedge\bar\varphi)=\lambda\wedge M(\bar\varphi)=\lambda\wedge \psi_X(M)(\varphi).$$
\end{proof}
\begin{prop} The map $\psi_X\colon J^2(X)\to J(X)$ is continuous.
\end{prop}
\begin{proof} Let $\mu=\psi_X(M)$, where $M\in J^2(X)$. Given $\varphi_1,\dots,\varphi_n\in C(X)$, the continuity follows from the formula
$$\psi_X(O( M; \bar\varphi_1,\dots,\bar\varphi_n;\varepsilon))\subset O( \mu; \varphi_1,\dots,\varphi_n;\varepsilon).$$

\end{proof}

Since the functors $J$ and $I$ are isomorphic, the notion of support is defined for $J$. Namely, given $\mu\in J(X)$, we say that the support of $\mu$ is the set $$\mathrm{supp}(\mu)=\bigcap\{A\in\exp X\mid \mu\in J(A)\subset J(X)\}\in\exp X.$$

Let $J_\omega(X)$ denote the set $\{\mu\in J(X)\mid \mathrm{supp}(\mu)\text{ is finite}\}$. By $J_\omega^2(X)$ we denote the set $\{\mu\in J(X)\mid \mathrm{supp}(\mu)\subset J_\omega(X)\}$. Similarly, let $$J_\omega^3(X)=\{\mu\in J(X)\mid \mathrm{supp}(\mu)\subset J_\omega^2(X)\}.$$

\begin{prop} $\psi=(\psi_X)$ is a natural transformation.
\end{prop}

\begin{proof}
 Let $f\colon X\to Y$ be a map. We are going to show that the diagram
$$\xymatrix{J^2(X)\ar[r]^{J^2(f)}\ar[d]_{\psi_X} & J^2(Y)\ar[d]^{\psi_Y}\\
J(X)\ar[r]_{J(f)} & J(Y)}$$
is commutative.

Since the set $J_\omega^2(X)$  is dense in $J(X)$, it is sufficient to verify the commutativity of the diagram for $\mu\in J_\omega^2(X)$. Let $M\in J_\omega^2(X)$, then there are $\alpha_1,\dots,\alpha_n\in[-\infty,\infty]$ and $\mu_1,\dots,\mu_n\in J_\omega(X)$ such that  $\vee_{i=1}^n\alpha_n=\infty$ and $M=\vee_{i=1}^n\alpha_i\wedge\delta_{\mu_i}$.

Then \begin{align*}\psi_YJ^2(f)(M)=&\psi_YJ^2(f)(\vee_{i=1}^n\alpha_i\wedge\delta_{\mu_i})=\psi_Y(\vee_{i=1}^n\alpha_i\wedge\delta_{J(f)(\mu_i)})=\vee_{i=1}^n\alpha_i\wedge J(f)(\mu_i)\\=&J(f)(\vee_{i=1}^n\alpha_i\wedge \mu_i)=J(f)\psi_X(M)\end{align*}
and we are done.
\end{proof}

\begin{thm}
The triple $\mathbb J=(J,\eta,\psi)$ is a monad on the category of compact metrizable spaces and continuous maps.
\end{thm}
\begin{proof}  Since the set $J_\omega^2(X)$ (resp.  $J_\omega^3(X)$) is dense in $J(X)$ (resp. $J^3(X)$), it is sufficient to verify the two-side unit property   $\psi_X T(\eta_X)(\mu)=\psi_X \eta_{T(X)}(\mu)=\mu$ for $\mu\in J_\omega^2(X)$, and the associativity property $\psi_X\psi_{T(X)}(\mu)=\psi_X T(\psi_X)(\mu)$ for $\mu\in J_\omega^3(X)$. Then we use the proof of \cite[Theorem 4.3]{CRZ} to complete the proof.
\end{proof}

The monad $\mathbb I=(I,\eta,\zeta)$ is defined in \cite{Z}. Note that $\eta_X(x)=\delta_x$, for $x\in X$. As for $\zeta_X$, in the sequel we will only need to know that $\zeta_X(\vee_{i=1}^n (\beta_i+\delta_{\mu_i}))= \vee_{i=1}^n (\beta_i+\mu_i)$, where $\beta_i\in[-\infty,0]$ and $\mu_i\in I(X)$, $i=1,\dots,n$ (actually, this property uniquely determines the map $\zeta_X$; see \cite{Z} for details).

\begin{thm} The monads $\mathbb J$ and $\mathbb I$ are not isomorphic.
\end{thm}
\begin{proof}
It is shown in \cite{CRZ} that every isomorphism between the functors $I_\omega$ and $J_\omega$ in the category of ultrametric spaces and nonexpanding maps is of a special form. Namely, for every such an isomorphism $k=(k_X)$ there exists an order-preserving bijection $\alpha \colon [-\infty,0]\to[-\infty,\infty]$ such that  $k_X(\vee_{i=1}^n(\lambda_i+\delta_{x_i}))=  \vee_{i=1}^n\alpha(\lambda_i)\wedge\delta_{x_i}   $. The proof of this fact is based on properties of the restrictions of the functors $I$ and $J$ onto the category of spaces of cardinality $\le 3$. Therefore, the corresponding fact is valid in the category $\mathbf{Comp}$.

Next, the calculations from the proof of \cite[Theorem 4.9]{CRZ} demonstrate that the monads $\mathbb J$ and $\mathbb I$ are not isomorphic. For the sake of reader's convenience, we repeat them below.

Let $X = \{a, b, c\}$. Suppose that $M = ((-1)+\delta_\mu)\vee \delta_\nu\in I^2_\omega
(X)$, where $\mu = ((-2)+\delta_a)\vee \delta_b$,
$\nu = ((-3)+\delta_b)\vee\delta_c$.
Then
$$k_X\zeta_X(M) =k_X (((-3)+\delta_a )\vee ((-3) +\delta_b) \vee\delta_c)
=(\alpha(-3) \wedge\delta_a) \vee( \alpha(-3) \wedge\delta_b)\vee \delta_c.$$
On the other hand,
\begin{align*}
&\psi_XJ(k_X )k_{I(X)}
(M) = \psi_XJ(k_X )(\alpha(-1)\wedge \delta_\mu\vee\delta_\nu)\\
=&\psi_X(\alpha(-1) \wedge\delta_{k_X(\mu)} \vee \delta_{k_X(\nu)}) = \psi_X(\alpha(-1) \wedge\delta_{\alpha(-2)\wedge\delta_a\vee\delta_c}
\vee \delta_{\alpha(-3)\wedge\delta_b\vee\delta_c}\\
=&\alpha(-2)\wedge\delta_a\vee \alpha(-3)\wedge\delta_b\vee\delta_c\neq k_X\zeta_X (M).
\end{align*}
\end{proof}

\section{Max-min convexity}

Let $A$ be a subset of $\mathbb R^\tau$. We say that $A$ is max-min convex if, for any $x,y\in A$ and $\lambda\in\mathbb R$, $x\vee(\lambda\wedge y)\in A$ (see, e.g., \cite{NS}). One can easily check that, if $A$ is a max-min convex set, then $\vee_{i=1}^n(\lambda_i\wedge x_i)\in A$ for any $x_1,\dots,x_n\in A$ and $\lambda_1,\dots,\lambda_n\in\mathbb R\cup\{\infty\}$ with $\vee\lambda_i=\infty$ (we express this by saying that every max-min convex set contains all max-min convex combinations of its elements).

Note that, for any $\mu,\nu\in J(X)$ and any $\lambda\in\mathbb R$, one can define $\mu \vee(\lambda\wedge \nu)\colon C(X)\to\mathbb R$ as follows:
$$(\mu \vee(\lambda\wedge \nu))(\varphi)= \mu (\varphi)\vee(\lambda\wedge \nu(\varphi)),\ \varphi\in C(X).$$
One can easily check that $\mu \vee(\lambda\wedge \nu)\in J(X)$. Actually, if one regards $J(X)$ as a subset in $\mathbb R^{C(X)}$, then $\mu \vee(\lambda\wedge \nu)$ is obtained by applying the operations $\vee$ and $\wedge$ coordinatewise. Therefore, we obtain the following
\begin{prop} The set $J(X)$ is a max-min convex subset in $\mathbb R^{C(X)}$.
\end{prop}

Now, let $A$ be a compact subset of $\mathbb R^\tau$. We are going to define a map $\xi\colon J(A)\to \mathbb R^\tau$ as follows. For any $\alpha<\tau$, let $p_\alpha\colon A\to\mathbb R$ denote the projection onto the $\alpha$-th coordinate, i.e. $p_\alpha((x_\beta)_{\beta<\tau})=x_\alpha$ for every $(x_\beta)_{\beta<\tau}\in A$. Then $p_\alpha\in C(A)$, $\alpha<\tau$, and we let $\xi(\mu)=(\mu(p_\alpha))_{\alpha<\tau}$, $\mu\in J(X)$. Clearly, $\xi$ is continuous and preserves the max-min convex combinations.

\begin{prop}\label{p:xi} If $A\subset \mathbb R^\tau$ is a compact max-min convex set, then $\xi(J(A))\subset A$.
\end{prop}
\begin{proof} First, note that, for any $x=(x_\alpha)_{\alpha<\tau}\in A$, $\delta_x(p_\alpha)=x_\alpha$ and therefore  $\xi(\delta_x)=x$. Since  $\xi$  preserves the max-min convex combinations, we conclude that $\xi(\mu) \in A$, for every $\mu\in J_\omega(A)$. Now, since $\xi$
is continuous and $J_\omega(A)$ is dense in $J(A)$, the assertion follows.
\end{proof}

The map $\xi$ is called the {\it max-min barycenter map}.

Given a monad $\mathbb T=(T,\eta,\mu)$ on a category $\mathcal C$, we say that a pair $(X,\xi)$ is a $\mathbb T$-algebra if $X$ is an object of $\mathcal C$ and $\xi\colon T(X)\to X$ is a morphism in $\mathcal C$ satisfying $\xi\eta_X=1_X$ and $\xi\mu_X=\xi T(\xi)$ (see \cite{BW}).

\begin{prop} For any compact subset $A\subset \mathbb R^\tau$, the pair $(A,\xi)$, where $\xi\colon J(A)\to A$ is the max-min barycenter map, is a $\mathbb J$-algebra.
\end{prop}
\begin{proof} We have already remarked (see the proof of Proposition \ref{p:xi}) that $\xi\delta_x=x$, for every $x\in A$.

Now, let $M\in J_\omega(J(X))$, $M=\vee_{i=1}^n(\alpha_i\wedge\delta_{\mu_i})$, where $\mu_1,\dots,\mu_n\in J (X)$, $\alpha_1,\dots,\alpha_n\in\mathbb R\cup\{\infty\}$ and $\vee_{i=1}^n\alpha_i=\infty$.

Then $$ \xi J(\xi)(M)=\xi(\vee_{i=1}^n(\alpha_i\wedge\delta_{\xi(\mu_i)}))=\vee_{i=1}^n(\alpha_i\wedge\xi(\mu_i))=\xi(\psi_X(M)).$$
Since the subset $J_\omega(J(X))$ is dense  in the space $J^2(X)$, we conclude that $(A,\xi)$ is a $\mathbb J$-algebra.
\end{proof}

\section{Remarks and open questions}

For any metric space $(X,d)$, the homeomorphism $h_X$ allows for metrization of the space $J(X)$ (as well as of $I(X)$; note that another metrization of $I(X)$ is considered in \cite{BRZ}). We will consider in details this metrization and its applications, in particular, to fractal geometry, in a separate publication.

 In \cite{Z} it is conjectured  (and proved in some cases) that the $\mathbb I$-algebras can be identified with the max-plus convex sets. Analogously, we conjecture that the $\mathbb J$-algebras can be naturally identified with the max-min convex sets.

 Recently, T. Radul \cite{Ra} considered the max-plus counterpart of the barycenter map. He characterized the compact max-plus convex sets for which the max-plus barycenter map is open. For the max-min convex sets the corresponding question remains unsolved.


\begin{thebibliography}{99}

\bibitem{AT} S. Ageev, E.D. Tymchatyn, On exact atomless Milutin maps,
 Topology and its Applications
Volume 153, Issues 2–3, 1 September 2005, Pages 227--238.

\bibitem{Ak} M. Akian,
Densities  of  idempotent  measures  and  large  deviations, Trans. Amer. Math. Soc. 351
(1999), no. 11, 4515--4543.

\bibitem{BW} M. Barr, Ch. Wells, (1985), Toposes, Triples and Theories, Grundlehren der mathematischen Wissenschaften, Springer-Verlag, 278,  1985,
Republished in:
Reprints in Theory and Applications of Categories, No. 12 (2005) pp. 1--287.

\bibitem{BRZ} L. Bazylevych, D. Repov\v s, M. Zarichnyi, Spaces of idempotent measures of compact metric spaces, Topol. Appl. 157 (2010), 136--144.

\bibitem{Be} G. Beer, Topologies on closed and closed convex sets. Kluwer Academic Publishers, 1994.

\bibitem{CRZ} Cencelj M., Repov\v s D., Zarichnyi M., Max-Min Measures on Ultrametric Spaces, Topology Appl., 160:5 (2013), 673--681.



\bibitem{Co}  J. B. Conway. A Course in Functional Analysis, Springer-Verlag, New. York, New York, 1990.
\bibitem{CGQS} G. Cohen,  S. Gaubert,  J. Quadrat, I. Singer,  Max-plus convex sets and functions. In: Litvinov, G.L., Maslov, V.P. (eds.): Idempotent Mathematics and Mathematical Physics. Contemporary Mathematics. American Mathematical Society, pp. 105--129. Also ESI Preprint 1341, arXiv:math.FA/0308166 (2005)

\bibitem{KN}  V.N. Kolokoltsov, and V.P. Maslov,
Idempotent Analysis and Applications,
Kluwer Academic Publishers, Dordrecht, 1997.

\bibitem{LM} G.L. Litvinov, and V.P. Maslov,
Correspondence Principle for Idempotent
Calculus and Some Computer Applications, Institut des Hautes Etudes Scientifiques, IHES/M/95/33, Bures-sur-Yvette, 1995; [6], pp. 420--443.



\bibitem{NS} V. Nitica, S. Sergeev,
Tropical convexity over max–min semiring, In:
Tropical and Idempotent Mathematics and Applications, Contemporary Mathematics, vol. 616, AMS, Providence (2014), pp. 241--260.
\bibitem{Ra} T. Radul, On the openness of the idempotent barycenter map,  arXiv:1706.06823

\bibitem{S} E.V. Shchepin. Functors and uncountable powers of compacta. Uspekhi Mat. Nauk \textbf{31} (1981), 3--62.

\bibitem{Su} M. Sugeno,
Theory of fuzzy integrals and its applications, PhD thesis,
Tokyo Institute of Technology, 1974.
\bibitem{Z} M. Zarichnyi. Spaces and maps of idempotent measures. Izvestiya: Mathematics, 2010, 74:3, 481--499.

\end{thebibliography}
\end{document}